\documentclass[12pt]{amsart}
\usepackage[utf8]{inputenc}

\font\myfont=cmr12 at 20pt

\title[Daubechies' localization operator on Cantor Type Sets]{\myfont{\textbf{Daubechies' Time-Frequency Localization Operator on Cantor Type Sets}}}
\author{Helge Knutsen}
\email{helge.knutsen@ntnu.no}
\date{July 2019}
\address{Department of Mathematical Sciences, Norwegian University of Science and \mbox{Technology,} \newline 7034 Trondheim, Norway}

\usepackage{amsthm}
\theoremstyle{plain} 

\newtheorem{theorem}{Theorem}[section]

\newtheorem{lemma}[theorem]{Lemma}
\newtheorem{proposition}{Proposition}[section]

\usepackage{amsmath, amssymb} 
\numberwithin{equation}{section} 

\usepackage{mathtools}

\DeclarePairedDelimiter\floor{\lfloor}{\rfloor} 

\usepackage{enumitem} 

\usepackage{calrsfs} 
\DeclareMathAlphabet{\pazocal}{OMS}{zplm}{m}{n}

\usepackage{bm} 

\usepackage[toc,page]{appendix} 

\usepackage{graphicx}
\graphicspath{ {figures/} }

\usepackage{pst-node}
\usepackage{tikz-cd}

\usepackage{geometry}
\usepackage{afterpage}

\usepackage{verbatim} 

\usepackage{subfigure}

\usepackage[hyphens]{url}

\begin{document}

\maketitle



\begin{abstract}
We study Daubechies' time-frequency localization operator, which is characterized by a window and weight function. We consider a Gaussian window and a spherically symmetric weight as this choice yields explicit formulas for the eigenvalues, with the Hermite functions as the associated eigenfunctions.
Inspired by the fractal uncertainty principle in the separate time-frequency representation, we define the $n$-iterate spherically symmetric Cantor set in the joint representation. 
For the $n$-iterate Cantor set, precise asymptotic estimates for the operator norm are then derived up to a multiplicative constant. 
\end{abstract}


\section{Introduction}
The problem of localizing signals in time and frequency is an old and important one in signal analysis. In applications, we often wish to analyze signals on different time-frequency domains, and we would therefore attempt to concentrate signals on said domains. Different approaches for how to construct such \textit{time-frequency localization operators} have been suggested, either based on a separate or joint time-frequency representation of the signal (see \cite{Prolate_Spheroidal_Wave_FunctionI}, \cite{DaubechiesI.1988Tloa}). The localization operators, regardless of which we choose to work with, will however be limited by the fundamental barrier of time-frequency analysis, namely the \textit{uncertainty principles}, which state that a signal cannot be highly localized simultaneously in \textit{both} time and frequency. With regard to the localization operator, the limits posed by the uncertainty principles translate into the associated \textit{operator norm}, as it measures the optimal efficiency of any given localization operator.    

Many versions of the uncertainty principles exist (see \cite{FollandGerald1997TupA}), and more recent versions start to take into account the \textit{geometry} of the time-frequency domains. In particular, in \cite{DyatlovFUP2019}, Dyatlov describes the
development and applications of a \textit{fractal uncertainty principle} (FUP) for the separate time-frequency representation. The relevant localization operator is the standard composition of projections $\pi_{T}Q_{\Omega}$, where $\pi_{T}$ and $Q_{\Omega}$ projects onto the sets $T$ in time and $\Omega$ in frequency, respectively. In the context of the FUP, the sets $T$ and $\Omega$ take the form of fractal sets or exhibit a regularity close to it. This regularity is, in part, represented as a sequence of sets $\{ X_n \}_n$, where $|X_n| \to \infty$ as $n\to \infty$. However, this sequence is constructed such that $\|\pi_{X_n}Q_{X_n}\|_{\text{op}}\to 0$. 

An illustrative example featured in \cite{DyatlovFUP2019} is a sequence based on the $n$-iterate mid-third Cantor set, defined in an ever increasing interval. More precisely, if $X_n$ denotes the $n$-iterate defined in the interval $[0,M]$, then the interval length satisfies
\begin{align}
    3^n \sim M^2,
    \label{Dyatlov_condition}
\end{align}
which means
$|X_n| \sim \left({2}/{\sqrt{3}}\right)^n \to \infty$ as $n\to \infty$. Further, by Theorem 2.13 in \cite{DyatlovFUP2019}, there exist constants $\alpha, \beta >0$ such that 
\begin{align}
    \| \pi_{X_n} Q_{X_n}\|_{\text{op}} \leq \alpha e^{-\beta n} \ \ \forall \ n \geq 0. 
    \label{Dyatlov_thm213}
\end{align}

Extending to the joint time-frequency representation we should expect some analogous result to the FUP (see Itinerary page 1 in \cite{GrochenigKarlheinz2001Fota}). Inspired by the model example, we search for similar results in the joint representation. In particular, we consider \textit{Daubechies' localization operator}, first introduced in \cite{DaubechiesI.1988Tloa}, based on the \textit{Short-Time Fourier tranform}, with a \textit{spherically symmetric} weight function and a Gaussian window. The reason for these restrictions is, as was shown in the aforementioned paper, that we obtain explicit expressions for the eigenvalues of the localization operator, with the Hermite functions as the associated eigenfunctions.

The remainder of the paper is organized as follows: In section \ref{section_preliminaries} we provide a more detailed introduction to the Daubechies operator (section \ref{section_Fourier_and_STFT}-\ref{section_Daubechies_loc_op}), in addition to some necessary results in the spherically symmetric context (section \ref{section_spherically_symmetric_weight}). We also make clear what we mean by a spherically symmetric Cantor set (section \ref{section_cantor_set}). New results are found in section \ref{section_main_results}, which contains several estimates for the operator norm of Daubechies' localization operator. 

After some preliminary examples, in section \ref{section_results_ring}-\ref{section_results_set_infinite_measure}, we finally consider the $n$-iterate spherically symmetric Cantor set in section \ref{section_cantor_set}. Here we derive \textit{precise} asymptotic estimates (up to a multiplicative constant) for the operator norm of the associated Daubechies' operator.
A particular case of this two-parameter result, in terms of the radius $R$ and iterate $n$, can be formulated as an estimate solely in terms of the parameter $n$. In the spherically symmetric context, we consider the condition
\begin{align}
    3^n \sim \left(\pi R^2 \right)^2,
    \label{condition_one_parameter_problem}
\end{align}
similar to condition \eqref{Dyatlov_condition}.
Hence, under the above condition, let $P_n$ denote the Daubechies operator localizing on the $n$-iterate spherically symmetric Cantor set defined in the disk of radius $R>0$. Then for some positive constants $c_1 \leq c_2$ the operator norm satisfies
\begin{align}
   c_1 \left(\frac{2}{3}\right)^{n/2}\leq \| P_n\|_{\text{op}} \leq c_2 \left(\frac{2}{3}\right)^{n/2}.
    \label{one_parameter_problem_result}
\end{align}
This result is analogous to knowing the exponential $\beta>0$ in \eqref{Dyatlov_thm213} precisely.


\section{Preliminaries}
\label{section_preliminaries}
\subsection{Fourier and Short-Time Fourier Transform}
\label{section_Fourier_and_STFT}
For a function
$f:\mathbb{R}\rightarrow\mathbb{C}$ the \textit{Fourier transform} evaluated at point $\omega\in \mathbb{R}$ is given by 
\begin{align*}
    \hat{f}(\omega) = \int_{\mathbb{R}}f(t)e^{-2\pi i \omega t}\mathrm{d}t.
\end{align*}
If we interpret $f$ as an amplitude signal depending on \textit{time}, then its Fourier transform $\hat{f}$ corresponds to a \textit{frequency} representation of the signal. The pair $(f,\hat{f})$ does not, however, offer a joint description with respect to both frequency and time. For this purpose, we consider the \textit{Short-Time Fourier transform} (STFT) (see Chapter 3 in \cite{GrochenigKarlheinz2001Fota}).

The STFT is often referred to as the "windowed Fourier transform" as this transform relies on an additional fixed, non-zero function, $\phi:\mathbb{R}\rightarrow \mathbb{C}$, known as a \textit{window function}. At point $(\omega,t)\in \mathbb{R}\times\mathbb{R}$ the STFT of $f$ with respect to the window $\phi$ is then defined as 
\begin{align*}
    V_{\phi}f(\omega,t) = \int_{\mathbb{R}}f(x)\overline{\phi(x-t)}e^{-2\pi i \omega x}\mathrm{d}x.
\end{align*}
The transformed signal now depends on both time $t$ and frequency $\omega$, and we refer to the $(\omega,t)$-domain $\mathbb{R}^2$ as the \textit{phase space} or the \textit{time-frequency plane}.

We will restrict our attention to signals and windows in $L^2(\mathbb{R})$, which, by Cauchy-Schwarz' inequality, ensures $V_{\phi}f(\omega,t)$ to be well-defined for all points $(\omega,t)\in \mathbb{R}^2$. Such restrictions also produce the following orthogonality relation
\begin{align}
    \langle V_{\phi_1}f_1, V_{\phi_2}f_2 \rangle_{L^2(\mathbb{R}^{2})}
    &=\langle f_1, f_2 \rangle \overline{\langle \phi_1, \phi_2 \rangle} \ \ \forall \ f_1,f_2,\phi_1,\phi_2 \in L^2(\mathbb{R}).
    \label{Orthogonality_relation}
\end{align}
Equipped with the standard $L^2$-norms, we deduce that the STFT is a bounded linear map onto the target space $L^2(\mathbb{R}^2)$. If the window $\phi$ is normalized, i.e. $\|\phi\|_2 =1$, then the STFT becomes, in fact, an \textit{isometry} onto some subspace of $L^2(\mathbb{R}^2)$.

Further, by identity \eqref{Orthogonality_relation}, the original signal $f$ can be recovered from its phase space representation. Take any $\gamma \in L^2(\mathbb{R})$ such that $\langle \gamma,\phi \rangle \neq 0$, then the orthogonal projection of $f$ onto any $g\in L^2(\mathbb{R})$ is given by
\begin{align*}
    \langle f, g \rangle = \frac{1}{\langle \gamma, \phi \rangle}\iint_{\mathbb{R}^{2}} V_{\phi}f(\omega,t) \overline{V_{\gamma}g(\omega,t)} \mathrm{d}\omega\mathrm{d}t.
\end{align*}
A canonical choice for $\gamma$ is to set it equal to $\phi$. Assuming $\phi$ is normalized, these projections then read
\begin{align}
    \langle f, g \rangle &= \iint_{\mathbb{R}^{2}}V_{\phi}f(\omega,t)\overline{V_{\phi}g(\omega,t)}\mathrm{d}\omega\mathrm{d}t.
    \label{inversion_formula_normalized}
\end{align}
Since any signal $f \in L^2(\mathbb{R})$ is entirely determined by such inner products, the right-hand side of formula \eqref{inversion_formula_normalized} provides a complete recovery from the STFT.

\subsection{Daubechies' Localization Operator}
\label{section_Daubechies_loc_op}
One approach for how to construct operators that localize a signal $f$ in both time and frequency was suggested by Daubechies in \cite{DaubechiesI.1988Tloa}. These operators can be summarized as modifying the STFT of $f$ by multiplication of a weight function, say $F(\omega,t)$, before recovering a time-dependent signal. The weight function aims at enhancing certain features of the phase space while diminishing others. 
Based on formula \eqref{inversion_formula_normalized}, we consider the sesquilinear functional
$\mathcal{P}_{F,\phi}$ on the product $L^2(\mathbb{R})\times L^2(\mathbb{R})$, defined by
\begin{align*}
    \mathcal{P}_{F,\phi}(f,g) = \iint_{\mathbb{R}^{2}}F(\omega,t) V_\phi f(\omega,t) \overline{V_\phi g(\omega,t)} \mathrm{d}\omega\mathrm{d}t. 
\end{align*}
Assuming $\mathcal{P}_{F,\phi}$ is a bounded functional, \textit{Riesz' representation theorem}
ensures the existence of a bounded, linear operator $P_{F,\phi}: L^2(\mathbb{R})\rightarrow L^2(\mathbb{R})$ such that
\begin{align*}
    \mathcal{P}_{F,\phi}(f,g) = \langle P_{F,\phi} f, g \rangle. 
\end{align*}
The operator $P_{F,\phi}$ is our sought after time-frequency localization operator, which we will refer to as \textit{Daubechies' localization operator}.
From the above defninition, $P_{F,\phi}$ is characterized by the choice of weight $F$ and window function $\phi$. 

In particular, any real-valued, integrable weight $F$ will produce self-adjoint, compact operators $P_{F,\phi}$ whose eigenfunctions form a complete basis for the space $L^2(\mathbb{R})$. Furthermore, the eigenvalues $\{\lambda_k\}_k$ satisfies $\sum_{k}|\lambda_k|\leq \|F\|_{1}$, in addition to $|\lambda_k|\leq \|F\|_{\infty}$ for all $k$.

\subsection{Spherically Symmetric Weight}
\label{section_spherically_symmetric_weight}
For an arbitrary weight $F$ and window $\phi$ it remains a challenge to determine the eigenvalues of Daubechies' localization operator $P_{F,\phi}$. However, in \cite{DaubechiesI.1988Tloa}, Daubechies narrows in her focus to operators with a normalized Gaussian window
\begin{align}
    \phi(x) = 2^{1/4}e^{-\pi x^2}, 
    \label{normalized_Gaussian}
\end{align}
and a spherically symmetric weight 
\begin{align}
    F(\omega, t) = \mathcal{F}(r^2)
    \label{def_spherically_symmetric_weight},
\end{align}
where $r^2 = \omega^2+t^2$.
For such operators, the Hermite functions\footnote{Due to the choice of normalization for the Fourier transform, both the Gaussian and the Hermite functions are normalized differently than in \cite{DaubechiesI.1988Tloa}. The normalization is chosen in accordance with Folland\cite{FollandGeraldB1989Haip}. If $h_k$ denotes the $k$-th Hermite function in \cite{DaubechiesI.1988Tloa}, this relates to $H_k$ in \eqref{Hermite_functions_one_dimensional} by $H_k(x) =\frac{2^{1/4}}{\sqrt{2^k k!}}h_k(\sqrt{2\pi}x)$.
} 
\begin{align}
    H_k(t) = \frac{2^{1/4}}{\sqrt{k!}}\left(-\frac{1}{2\sqrt{\pi}}\right)^{k}e^{\pi t^2}\frac{\mathrm{d}^k}{\mathrm{d}t^k}(e^{-2\pi t^2}),\ k=0,1,2,\dots
    \label{Hermite_functions_one_dimensional}
\end{align}
are shown to constitute the eigenfunctions. Further, explicit expressions for the associated eigenvalues $\{\lambda_k\}_k$ are derived.

\begin{theorem}
\textnormal{(Daubechies)} Let $P_{F,\phi}$ denote the localization operator with weight $F(\omega, t) = \mathcal{F}(r^2)$ and window $\phi$ equal to the normalized Gaussian in \eqref{normalized_Gaussian}. Then the eigenvalues of $P_{F,\phi}$ are given by
\begin{align*}
    \lambda_k = \int_{0}^{\infty}\mathcal{F}\left(\frac{r}{\pi}\right)\frac{r^{k}}{k!}e^{-r}\mathrm{d}r, \ \text{for } k=0,1,2,\dots, 
\end{align*}
such that
\begin{align*}
    P_{F,\phi}H_k = \lambda_k H_k, 
\end{align*}
where $H_k$ denotes the $k$-th Hermite function.
\label{theorem_derivation_spectrum_gaussian_window}
\end{theorem}
Observe that the normalized Gaussian in \eqref{normalized_Gaussian} coincides with $H_0$ in \eqref{Hermite_functions_one_dimensional}. It was shown recently in \cite{Weyl_Heisenberg_Ensembles_Grochenig_2017} that (for each $j$) the Hermite functions are also eigenfunctions of any localization operator with window $H_j$ and a spherically symmetric weight. Nevertheless, we will always presume the window $\phi$ to be the normalized Gaussian.  

We will consider the case when $\mathcal{F}$ equals the characteristic function of some subset $E\subseteq \mathbb{R}_{+}$, i.e., $\mathcal{F}(r) = \chi_{E}(r)$. Note that by definition \eqref{def_spherically_symmetric_weight}, the set $E$ is identified with the subset
$\pazocal{E} = \{ (\omega,t) \in \mathbb{R}^2 \ | \ \omega^2+t^2 \in E\}$ of the plane. As a matter of convenience, we will denote the associated Daubechies operator simply by $P_{E}$.
By Theorem \ref{theorem_derivation_spectrum_gaussian_window}, the eigenvalue corresponding to the $k$-th Hermite function is then given by 
\begin{align}
    \lambda_{k} = \int_{\pi\cdot E}\frac{r^k}{k!}e^{-r}\mathrm{d}r, \ \text{for } k = 0,1,2,\dots,
    \label{eigenvalues_characteristic_function}
\end{align}
where $\pi \cdot E := \{ x \in \mathbb{R}_{+}\ | \ x/\pi \in E\}$. 
Since the above integrands will appear frequently, we define, for simplicity, the functions
\begin{align*}
    f_k(r) := \frac{r^k}{k!}e^{-r}, \ r\geq 0, \ \text{for } k=0,1,2,\dots
\end{align*}

In section \ref{section_localization_on_cantor_set} we require two results regarding the integrands $\{f_k\}_k$ (for additional details, see Chapter 4.2 in \cite{HelgeKnutsen_master_thesis}), namely
\begin{align}
    f_k(k-r) \leq f_k(k+r) \ \ \forall \ r\in[0,k] \ \text{for } k=1,2,3,\dots
    \label{right_slope_less_than_left_slope_f_k}
\end{align}
and 
\begin{align}
    \int_{0}^{|E|}f_0(r)\mathrm{d}r = 1-e^{-|E|} \geq \int_{E}f_k(r)\mathrm{d}r \ \ \text{for } k = 0,1,2,\dots, 
    \label{maximal_localization_disk}
\end{align}
where $E$ is some measurable subset of $\mathbb{R}_{+}$.

\subsection{Cantor Set}
\label{section_cantor_set}

The \textit{mid-third Cantor set} based in the interval $[0,R]$ is constructed as follows: Start with the interval $C_0(R)=[0,R]$. Each $n$-iterate $C_n(R)$ is the union of $2^n$ disjoint, closed intervals $\{I_{j,n}\}_{j}$. To obtain the next iterate $C_{n+1}(R)$  remove the open middle-third interval in every interval  $I_{j,n}$. 
Such iterations yield a nested sequence $C_0\supseteq C_1 \supseteq C_2 \supseteq \dots$
The \textit{mid-third Cantor} set $C(R)$ on the interval $[0,R]$ is then defined as the intersection of all the $n$-iterates, i.e.,
\begin{align*}
    C(R) = \bigcap_{n=0}^{\infty}C_n(R). 
\end{align*}

For each $n$-iterate, we define a corresponding map $\pazocal{G}_{R,n}: \mathbb{R}\to [0,1]$ by 
\begin{align}
    \pazocal{G}_{R,n}(x) = \frac{1}{|C_n(R)|}\cdot \begin{cases} 
    0, \ x\leq 0,\\
    |C_n(R)\cap[0,x]|, \ x > 0 \ \ \text{for } n=0,1,2,\dots,
    \end{cases}
    \label{cantor_function}
\end{align}
which we refer to as the $n$-iterate Cantor function. These functions will come into play in the latter part of section \ref{section_localization_on_cantor_set}, where we will utilize the fact that $\{\pazocal{G}_{R,n}\}_n$ are all \textit{subadditive}, i.e.,
\begin{align}
    \pazocal{G}_{R,n}(a+b) \leq \pazocal{G}_{R,n}(a)+\pazocal{G}_{R,n}(b) \ \ \forall \ a,b\in \mathbb{R},
    \label{subadditivity_cantor_function}
\end{align}
which was shown by induction by Josef Dobo\v{s} in \cite{JosefDobos_1996}.
\\\\
In the spherically symmetric context, we consider the following Cantor set construction: For the disk of radius $R>0$ centered at the orgin, we identify the $n$-iterate with the subset 
\begin{align}
    \mathcal{C}_n(R) = \{ (\omega,t)\in \mathbb{R}^2 \ | \ \omega^2+t^2 \in C_n(R^2) \} \subseteq \mathbb{R}^2.
    \label{n_iterate_measure_regular_cantor_set_disk}
\end{align}
This means we consider weights of the form
\begin{align*}
    \mathcal{F}(r) = \chi_{C_n(R^2)}(r), \ \text{for } R>0 \ \text{and } n=0,1,2,\dots
\end{align*}
Based on formula \eqref{eigenvalues_characteristic_function}, the eigenvalues of $P_{C_n(R^2)}$ can then be expressed as
\begin{align}
    \lambda_k(\mathcal{C}_n(R)) &= \int_{\pi \cdot C_n(R^2)}f_k(r)\mathrm{d}r = \int_{C_n(\pi R^2)}f_k(r)\mathrm{d}r.
    \label{eigenvalue_n_iterate_cantor_set}
\end{align}


\section{Main Results}
\label{section_main_results}
In this section we present estimates for the operator norm of Daubechies' operator localizing on different spherically symmetric sets. For this purpose, it would be sufficient to determine the largest eigenvalue of the operator and estimate said eigenvalue. Nonetheless, even with identity \eqref{eigenvalues_characteristic_function}, it may prove difficult to determine which eigenvalue is the largest. Under such circumstances, we will instead attempt to derive a common upper bound for the eigenvalues.

\subsection{Localization on a Ring: Asymptotic Estimate}
\label{section_results_ring}
The first example we consider shows that any eigenvalue $\lambda_k$ of Daubechies' localization operator can, in principle, be the largest eigenvalue. Consider localization on a ring of inner radius $R>0$ in phase space of measure $1$, that is, the subset $E(R) = [R^2, R^2+\pi^{-1}]$ with the associated localization operator $P_{E(R)}$. By \eqref{eigenvalues_characteristic_function}, the eigenvalues of $P_{E(R)}$ become 
\begin{align*}
    \lambda_k(R) = \int_{\pi E(R)}f_k(r)\mathrm{d}r = \int_{\pi R^2}^{\pi R^2 +1} f_k(r)\mathrm{d}r \ \ \text{for } k=0,1,2,\dots
\end{align*}
Now, assume that $\pi R^2 \in [m,m+1]$ for some $m\in \mathbb{N}\cup \{0\}$. Since the difference $f_k(r) - f_{k+1}(r)$ is negative precisely when $r>k+1$, we obtain the ordering
\begin{align*}
    &\lambda_0(R) \leq \lambda_1(R) \leq \lambda_2(R) \leq \dots \leq \lambda_m(R)
    \intertext{and}
    & \lambda_{m+1}(R) \geq \lambda_{m+2}(R) \geq \lambda_{m+3}(R) \geq \dots
\end{align*}
Under these conditions, either $\lambda_{m}(R)$ or $\lambda_{m+1}(R)$ must be the largest eigenvalue. In particular, if $\pi R^2 = m$, then $\lambda_m(R)$ becomes the largest eigenvalue.
In the next proposition we provide an asymptotic estimate of the operator norm of $P_{E(R)}$ as $R\to \infty$. 
\begin{proposition} 
The operator norm of $P_{[R^2, R^2+\pi^{-1}]}$ satisfies
\begin{align*}
    \| P_{[R^2, R^2 + \pi^{-1}]} \|_{\text{op}} = \frac{1}{\pi \sqrt{2}}R^{-1} + \pazocal{O}(R^{-3}) \ \ \text{as } R\to \infty.
\end{align*}
\begin{proof}
Assume $\pi R^2 \geq 1$ and let $n:= \floor{\pi R^2}$, where $\floor{\cdot}$ denotes the floor function, rounding down to the nearest integer. Apply a zero-order approximation (i.e., max-min) of the integrands $f_k(r)$ for $r\in[n,n+2] \supseteq [\pi R^2, \pi R^2+1]$. In particular, $f_n(n)$ serves as an upper bound and, by inequality \eqref{right_slope_less_than_left_slope_f_k}, $f_{n+1}(n)$ serves as a lower bound for the operator norm. That is, 
\begin{align*}
    \frac{n^{n+1}}{(n+1)!} e^{-n} \leq \| P_{[R^2,R^2+\pi^{-1}]}\|_{\text{op}} \leq \frac{n^n}{n!} e^{-n}.
\end{align*}
Combine this with \textit{Stirling's asymptotic formula} for the factorial
\begin{equation*}
    \sqrt{2\pi}\cdot n^{n+1/2}e^{-n} \leq n! \leq e^{\frac{1}{12n}} \sqrt{2\pi}\cdot n^{n+1/2}e^{-n} \ \text{for } n=1,2,3,\dots,
    \label{Stirlings_formula}
\end{equation*}
to obtain
\begin{align*}
    \frac{1}{\sqrt{2\pi}} n^{-1/2}\left(1+\frac{1}{n}\right)^{-1}e^{-\frac{1}{12n}} &\leq \| P_{[R^2, R^2+\pi^{-1}]}\|_{\text{op}} \leq \frac{1}{\sqrt{2\pi}}n^{-1/2}.
\end{align*}
Expressing the above identity in terms of $R$ as $R\to \infty$ yields the desired result.\\ 
\end{proof}
\end{proposition}
\subsection{Localization on Set of Infinite Measure}
\label{section_results_set_infinite_measure}
Next, we consider a non-trivial example of localization on a spherically symmetric set of \textit{infinite} measure (see \cite{ReznikovAlexander2010Scit} for a similar example in the separate time-frequency representation). Take the set of equidistant intervals 
\begin{align}
    E(s) := \bigcup_{n=0}^{\infty}\frac{1}{\pi}\cdot[n,n+s] \ \ \text{for } s \in [0,1]. 
    \label{equidistant_invervals_spherical_symmetric_set}
\end{align}
Although the above set has infinite measure, we maintain good control over the operator norm of $P_{E(s)}$ and can produce precise estimates in terms of the parameter $s$. 
\begin{theorem}
Let $E(s)\subseteq \mathbb{R}_{+}$ be as in \eqref{equidistant_invervals_spherical_symmetric_set} with $s\in[0,1]$. Then the operator norm of $P_{E(s)}$ satisfies the bounds
\begin{align}
    (1-e^{-s})C \leq \| P_{E(s)}\|_{\text{op}} \leq \min\{Cs,1\} \ \ \forall \ s\in[0,1] \ \ \text{with } C = \frac{e}{e-1}.
    \label{equidistant_intervals_bounds_operator_norm}
\end{align}
Further, there exists $s_0>0$ such that 
\begin{align}
    \|P_{E(s)}\|_{\text{op}} = (1-e^{-s})C \ \ \forall \ 0<s<s_0. 
    \label{equidistant_intervals_lambda_0=operator_norm}
\end{align}
\end{theorem}
\begin{proof}
By formula \eqref{eigenvalues_characteristic_function}, the eigenvalues read
\begin{align*}
    \lambda_k(s) = \int_{\pi\cdot E(s)}f_k(r)\mathrm{d}r = \sum_{n=0}^{\infty}\int_{n}^{n+s}f_k(r)\mathrm{d}r \ \ \text{for } k=0,1,2,\dots
\end{align*}
For each integral over $[n,n+s]$ apply a zero-order approximation for the integrands $f_k$, i.e., consider the maximum of $f_k(r)$ for $r\in[n,n+1]$ such that
\begin{align}
    \lambda_0(s) &\leq s\sum_{n=0}^{\infty}f_0(n) = s \sum_{n=0}^{\infty}e^{-n} = \frac{s}{1-e^{-1}}= Cs
    \label{equidistant_intervals_lambda_0_upper_bound}
    \intertext{and}
    \lambda_k(s) &\leq s\left( f_k(k) + \sum_{n=0}^{\infty}f_k(n)\right) \ \ \text{for } k=1,2,3,\dots
    \label{equidistant_intervls_lambda_k_upper_boud}
\end{align}
We now claim that the following inequality holds
\begin{align}
    f_k(k) + \sum_{n=0}^{\infty}f_k(n) < \sum_{n=0}^{\infty}f_0(n) \ \ \text{for } k=1,2,3,\dots
    \label{equidistant_intervals_sum_f_k_inequality}
\end{align}
For $k=1$, inequality \eqref{equidistant_intervals_sum_f_k_inequality} is verified by computing the series explicitly.
While for $k>1$, compare the series with the integral over $\mathbb{R}_{+}$, that is
\begin{align*}
    \sum_{n\neq k}f_k(n) \leq \int_{0}^{\infty}f_k(r)\mathrm{d}r = 1. 
\end{align*}
Thus,
\begin{align*}
    f_k(k) + \sum_{n=0}^{\infty}f_k(n) \leq 1+2f_k(k) \leq 1+2f_2(2) = 1+4e^{-2} \ \ \text{for } k=2,3,4,\dots
\end{align*}
Since $1+4e^{-2} < C$, claim \eqref{equidistant_intervals_sum_f_k_inequality} follows. Combining results \eqref{equidistant_intervals_lambda_0_upper_bound}-\eqref{equidistant_intervals_sum_f_k_inequality} yields the upper bound in \eqref{equidistant_intervals_bounds_operator_norm}. In the lower bound case of \eqref{equidistant_intervals_bounds_operator_norm}, it is sufficient to observe
\begin{align*}
    \lambda_0(s) = \sum_{n=0}^{\infty}\int_{n}^{n+s}e^{-r}\mathrm{d}r = (1-e^{-s})\sum_{n=0}^{\infty}e^{-n} = (1-e^{-s})C. 
\end{align*} 
For the equality case \eqref{equidistant_intervals_lambda_0=operator_norm}, note that inequlity \eqref{equidistant_intervals_sum_f_k_inequality} ensures that there exists a constant $0<C_0<C$ such that
$\lambda_k(s) \leq C_0s$ for any $k,s>0$. Since $1-e^{-s}\nearrow s$ as $s \to 0$, it follows that some $s_0>0$ with property \eqref{equidistant_intervals_lambda_0=operator_norm} exists.

\end{proof}

\subsection{Localization on Spherically Symmetric Cantor Set}
\label{section_localization_on_cantor_set}
In this section we consider localization on the $n$-iterate spherically symmetric Cantor set, i.e., the set $\mathcal{C}_n(R)$ in \eqref{n_iterate_measure_regular_cantor_set_disk}. Hence, we consider the localization operator $P_{C_n(R^2)}$, and below two theorems for the operator norm are presented.

The first theorem shows to what extent the operator norm is bounded by the first eigenvalue $\lambda_0(\mathcal{C}_n(R))$. 

\begin{theorem}
The operator norm of $P_{C_n(R^2)}$ is bounded from above by 
\begin{align*}
    \| P_{C_n(R^2)}\|_{\text{op}} \leq 2\lambda_0(\mathcal{C}_n(R)) \ \ \text{for } n = 0,1,2,\dots
\end{align*}
\label{bound_operator_norm_thm1}
\end{theorem}
The second theorem is a \textit{precise} asymptotic estimate of the operator norm of $P_{C_n(R^2)}$ (up to a multiplicative constant) based on the same asymptotic estimate for $\lambda_0(\mathcal{C}_n(R))$. 

\begin{theorem}
There exist positive, finite constants $c_1 \leq c_2$ such that for each $n=0,1,2,\dots$
\begin{align*}
    c_1 \leq \frac{\Big(2\pi R^2+1\Big)^{\frac{\ln{2}}{\ln{3}}}}{2^n\big({1-e^{-\pi R^2/3^n}}\big)}\cdot \| P_{C_n(R^2)}\|_{\text{op}}
    \leq c_2 \ \ \forall \ \pi R^2\in[0,3^n/2].
\end{align*}
\label{asymptotic_bound_operator_norm_thm2}
\end{theorem}
Observe that once we enforce condition \eqref{condition_one_parameter_problem}, namely $3^n \sim (\pi R^2)^2$, result \eqref{one_parameter_problem_result} presented in the introduction follows as a corollary. 
Both theorems are attained from the integral formula \eqref{eigenvalue_n_iterate_cantor_set} for the eigenvalues $\{\lambda_k(\mathcal{C}_n(R))\}_k$. However, as the number of intervals in $C_n(\cdot)$ grows as $2^n$, it soon becomes rather impractical to evaluate these integrals \textit{directly}. 
\\\\
Instead we consider the effect on the integrals \textit{locally} of increasing from one iterate to the next. In particular, this means we initially consider the integral of $f_k$ over a \textit{single} interval, say $[s,s+3L]$ for $s\geq 0$ and $L>0$. Then we attempt to determine the relative area left under the curve $f_k$ once the mid-third of the interval is removed, i.e., we wish to understand the function
\begin{align}
    \pazocal{A}_k(s,3L) := \left[\int_{s}^{s+L}f_k(r)\mathrm{d}r+\int_{s+2L}^{s+3L}f_k(r)\mathrm{d}r\right]\cdot \left[\int_{s}^{s+3L}f_k(r)\mathrm{d}r\right]^{-1}.
    \label{relative_area_left_definition}
\end{align}
Computing the above integrals, $\pazocal{A}_k(s,3L)$ can alternatively be expressed
\begin{align}
    \pazocal{A}_k(s,3L) = &\left[\sum_{n=0}^{k}\frac{1}{n!}\big(s^n-e^{-L}(s+L)^n+e^{-2L}(s+2L)^n-e^{-3L}(s+3L)^n\big)\right]\nonumber\\
    &\cdot \left[\sum_{n=0}^{k}\frac{1}{n!}\big(s^n-e^{-3L}(s+3L)^n\big)\right]^{-1}.
    \label{relative_area_left_identity}
\end{align}
Observe that $\pazocal{A}_k(s,3L)$ is independent of the starting point $s$ precisely when $k=0$. In particular, 
\begin{align}
    \pazocal{A}_0(3L):= \pazocal{A}_0(s,3L) =  \frac{\big(1+e^{-2L}\big)\big(1-e^{-L}\big)}{1-e^{-3L}}.
    \label{relative_area_left_lambda_0_alternative_identity}
\end{align}
For this reason, calculations with regard to $\lambda_0(\mathcal{C}_n(R))$ are significantly simpler than for the remaining eigenvalues.

We begin by computing asymptotic estimates for the eigenvalue $\lambda_0(\mathcal{C}_n(R))$ in section \ref{estimates_for_the_first_eigenvalue}. In section \ref{common_upper_bound_for_the_eigenvalues} we utilize the relative areas $\pazocal{A}_k(s,L)$ to determine a common upper bound for the eigenvalues. Here, Lemma \ref{lemma_non_shifted_iterates_subadditivity} is noteworthy as it relies on the subadditivity of the Cantor function.

\subsubsection{Estimates for the First Eigenvalue $\lambda_0(\mathcal{C}_n(R))$}
\label{estimates_for_the_first_eigenvalue}
Due to the fact that the relative area $\pazocal{A}_0(s,3L)$ is independent of the starting point $s$, we obtain the recursive relation
\begin{align*}
    \lambda_{0}(\mathcal{C}_{n+1}(R)) = \pazocal{A}_{0}(\pi R^2/3^{n}) \lambda_{0}(\mathcal{C}_{n}(R)) \ \text{for } n=0,1,2,\dots,
\end{align*}
which in return means
\begin{align}
    \lambda_{0}(\mathcal{C}_{n}(R)) &= \lambda_0(\mathcal{C}_0(R))\prod_{j=0}^{n-1}\pazocal{A}_0(\pi R^2/3^j)\nonumber\\
    &= \left(1-e^{-\pi R^2}\right)\prod_{j=0}^{n-1}\pazocal{A}_0(\pi R^2/3^j). \ \ \ \ \ \ \ \ \ \ \ \ \ \ \ \ \ \ \ 
    \label{eigenvalue_lambda_0_cantor_set_rel_prod}
\end{align}
From here we are able to formulate a precise asymptotic estimate for the first eigenvalue.

\begin{proposition}
There exist positive, finite constants $c_1 \leq c_2$ such that for each $n=0,1,2,\dots$
\begin{align*}
    c_1 \leq \frac{\Big(2\pi R^2+1\Big)^{\frac{\ln{2}}{\ln{3}}}}{2^n\big({1-e^{-\pi R^2/3^n}}\big)}\cdot \lambda_0(\mathcal{C}_n(R))
    \leq c_2 \ \ \forall \ \pi R^2\in[0,3^n/2].
\end{align*}
\begin{proof}
Combine the two identities \eqref{relative_area_left_lambda_0_alternative_identity}, \eqref{eigenvalue_lambda_0_cantor_set_rel_prod} to obtain
\begin{align*}
    \lambda_0(\mathcal{C}_n(R)) = \left(1-e^{-\pi R^2/3^n}\right)\prod_{j=1}^{n}\left(1+e^{-2\pi R^2/3^j}\right) \ \text{for } n=0,1,2,\dots
\end{align*}
By the latest result, it is sufficient to show that 
\begin{align*}
    c_1\leq \Big(2\pi R^2+1\Big)^{\frac{\ln{2}}{\ln{3}}} \prod_{j=1}^{n}\frac{1}{2}\left(1+e^{-2\pi R^2/3^j}\right)\leq c_2 \ \ \forall \  \pi R^2 \in [0,3^n/2].
\end{align*}
Exchange the product for a sum, and the above inequality is equivalent to 
\begin{align*}
    \ln{c}_1 \leq \sum_{j=1}^{n}\ln\left(1+e^{-x/3^j}\right)-\left(n - \frac{\ln(x+1)}{\ln 3}\right)\ln 2 \leq \ln{c}_2 \ \forall \ x \in [0,3^n].
\end{align*}
\noindent The last inequality relies on two claims 
\begin{enumerate}[label =  (\roman*),itemsep=0.4ex, before={\everymath{\displaystyle}}]
    \item 
    $\sup_{y\in[0,1]}\sum_{j=1}^{\infty}\left[\ln\left(1+y^{1/3^j}\right)-y^{1/3^j}\ln(2)\right] < \infty$, and
    \item there exist finite constants $\gamma_1\leq \gamma_2$ such that
    \begin{align*}
        \gamma_1 \leq \sum_{j=1}^{n}e^{-x/3^j}-\left(n-\frac{\ln(x+1)}{\ln(3)}\right) \leq \gamma_2 \ \text{for } x \in [0,3^n].\ \ \ \ \ \ \ \ \ \ \ \ \ \ \ \ \ \ \ \ 
    \end{align*}
\end{enumerate}
Although each claim is not difficult to verify, precise arguments are somewhat technical (see Appendix B in \cite{HelgeKnutsen_master_thesis} for details).

\end{proof}
\label{proposition_asymptotic_estimates_lambda_0}
\end{proposition}
\subsubsection{Common Upper Bound for the Eigenvalues}
\label{common_upper_bound_for_the_eigenvalues}
Next, we search for a common upper bound for the eigenvalues expressed in terms of the first eigenvalue. Note that if all the relative areas $\pazocal{A}_k(s,L)$ were bounded by $\pazocal{A}_0(L)$ regardless of starting point $s>0$ and interval length $L>0$, we would conclude that $\lambda_0(\mathcal{C}_n(R))$ is always the largest eigenvalue. As it turns out, this is not the case, e.g.,
\begin{align*}
    \lim_{L\to 0}\pazocal{A}_k(0,L)>\pazocal{A}_0(L) \ \ \text{for } k=2,3,4,\dots
\end{align*}
Instead, we compare the relative areas $\pazocal{A}_k(s,3L)$ and $\pazocal{A}_0(3L)$ when $s\geq k$. 
\begin{lemma}
Let $\{\pazocal{A}_k\}_k$ be given by \eqref{relative_area_left_definition}. Then 
\begin{align*}
    \pazocal{A}_k(s,3L)\leq \pazocal{A}_0(3L) \ \ \forall \ s\geq k, \ L>0\ \text{and } k=0,1,2,\dots
\end{align*}
\begin{proof}
Consider the derivative of $\pazocal{A}_k(s,L)$ with respect to $s$, which yields
\begin{align*}
    \frac{\partial \pazocal{A}_k}{\partial s}(s,3L)= N_k(s,L) \left[\int_{s}^{s+3L}f_k(r)\mathrm{d}r\right]^{-2},
\end{align*}
for some function
\begin{align*}
    N_k(s,L) = \Big(f_k(s+L)-f_k(s+2L)\Big)\int_{s}^{s+3L}f_k(r)\mathrm{d}r&\\
    -\Big(f_k(s)-f_k(s+3L)\Big)\int_{s+L}^{s+2L}f_k(r)\mathrm{d}r&.
\end{align*}
By identity \eqref{relative_area_left_identity}, it is clear that $\lim_{s\to \infty}\pazocal{A}_k(s,3L) = \pazocal{A}_0(3L)$ for all $L>0$. Thus, it suffices to show that $N_k(s,L)\geq 0$ for all $s\geq k$ and $L>0$ for $k=1,2,3,\dots$
Introduce the function 
\begin{align}
    \Phi_k(r,s,L) := &\Big[f_k(r+s)f_k(s+L)-f_k(r+s+L)f_k(s)\Big]\nonumber\\
    &+\Big[f_k(r+s+L)f_k(s+L)-f_k(r+s)f_k(s+2L)\Big].
    \label{proposition_phi_def}
\end{align}
Then we may express $N_k(s,L)$ as a single integral over $[0,L]$ such that
\begin{align*}
    N_k(s,L) = \int_{0}^{L}\Big(\Phi_k(r,s,L)-\Phi_k(r,s+L,L)\Big)\mathrm{d}r.
\end{align*}
Hence, the function $N_k(s,L)$ is \textit{positive} for all $s\geq k$ if the derivative of $\Phi(r,s,L)$ with respect to $s$ is \textit{negative}. Consider each of the square bracket terms $[\dots]$ in definition \eqref{proposition_phi_def} separately, that is 
\begin{align*}
    \Psi_k(r,s,L,y) := &f_k(r+s+y)f_k(s+L)\nonumber\\
    &-f_k(r+s+L-y)f_k(s+2y) \ \text{for} \ y \in \{0,L\},
\end{align*}
such that $\Phi_k(r,s,L) = \Psi_k(r,s,L,0)+\Psi_k(r,s,L,L)$. For each $y\in\{0,L\}$ the desired derivative properties of $\Psi_k$ are easy to determine. This is partly due the fact that the arguments of $f_k(\cdot)$ in each term of $\Psi_k$ add to a \textit{fixed} value. Computations then yield   
\begin{align*}
    \frac{\partial \Psi_k}{\partial s}(r,s,L,y) \leq 0 \ \ \forall \ s\geq k \ \text{and } y \in \{0,L\}    
\end{align*}
(see Appendix C in \cite{HelgeKnutsen_master_thesis} for details).

\end{proof}
\end{lemma}

Thus, for any shifted $n$-iterate Cantor set $C_n(\pi R^2)+s+k$ for $s\geq 0$, we obtain the recursice inequality
\begin{align*}
    \int_{C_{n+1}(\pi R^2)+s+k}f_k(r)\mathrm{d}r\leq \pazocal{A}_0(\pi R^2/3) \int_{C_n(\pi R^2)+s+k}f_k(r)\mathrm{d}r \ \ \text{for } n=0,1,2,\dots
\end{align*}
Furthermore, by result \eqref{maximal_localization_disk} and \eqref{eigenvalue_lambda_0_cantor_set_rel_prod},
\begin{align}
    \int_{C_n(\pi R^2)+s+k}f_k(r)\mathrm{d}r &\leq \int_{C_0(\pi R^2)+s+k}f_k(r)\mathrm{d}r \prod_{j=0}^{n-1}\pazocal{A}_0(\pi R^2/3^j)\nonumber\\
    &\leq \lambda_0(\mathcal{C}_0(R)) \prod_{j=0}^{n-1}\pazocal{A}_0(\pi R^2/3^j) = \lambda_0(\mathcal{C}_n(R)).
    \label{shifted_integral_cantor_set_inequality}
\end{align}
In the next lemma we relate the integrals of $f_k$ over the shifted $n$-iterate Cantor sets to the non-shifted iterates.
\begin{lemma}
Let $L>0$. Then for every fixed $k,n=0,1,2,\dots$, we have 
\begin{enumerate}[label =  \textnormal{(\Alph*)},itemsep=0.4ex, before={\everymath{\displaystyle}}]
    \item $\int_{C_n(L)\cap[k,\infty[}f_k(r)\mathrm{d}r\leq \int_{C_n(L)+k}f_k(r)\mathrm{d}r\ $ and
    \item $\int_{C_n(L)\cap [0,k]}f_k(r)\mathrm{d}r\leq \int_{C_n(L)+k}f_k(r)\mathrm{d}r$.
\end{enumerate}
\begin{proof}
For case (A), since $f_k(r)$ is monotonically decreasing for $r>k$, it suffices to verify
\begin{align*}
    |C_n(L)\cap[k,r]|\leq |(C_n(L)+k)\cap[k,r]| \ \ \forall \ r\geq k.
\end{align*}
In terms of the Cantor function $\pazocal{G}_{L,n}$ in \eqref{cantor_function}, the above claim reads
\begin{align*}
    \pazocal{G}_{L,n}(r)-\pazocal{G}_{L,n}(r) \leq \pazocal{G}_{L,n}(r-k) \ \ \forall \ r\geq k,
\end{align*}
which is the same subadditivity property as in \eqref{subadditivity_cantor_function}.

For case (B), consider the reflection of elements $C_n(L)\cap[0,k]$ about the point $k$, that is, consider the subset
\begin{align}
    \pazocal{R}_{n,k} := \{ r\geq k \ | \ 2k-r \in C_n(L)\cap[0,k]\},
    \label{reflected_set}
\end{align}
By result \eqref{right_slope_less_than_left_slope_f_k}, we have that
\begin{align*}
    \int_{C_n(L)\cap[0,k]}f_k(r)\mathrm{d}r \leq \int_{\pazocal{R}_{n,k}}f_k(r)\mathrm{d}r.
\end{align*}
Similarly to (A), it suffices to show that 
\begin{align}
    | \pazocal{R}_{n,k}\cap[k,r]| \leq |(C_n(L)+k)\cap[k,r]| = L\cdot \pazocal{G}_{n,L}(r-k) \ \ \forall \ r\geq k. 
    \label{lemma_cantor_set_start_at_k_bound_B_sufficient}
\end{align}
By definition \eqref{reflected_set}, the set $\pazocal{R}_{n,k}$ satisfies
\begin{align*}
    |\pazocal{R}_{n,k}\cap[k,r]| &= |C_n(L) \cap [2k-r,k]| = L \big( \pazocal{G}_{L,n}(k) - \pazocal{G}_{L,n}(2k-r)\big).
    \end{align*}
Apply subadditivity of $\pazocal{G}_{L,n}$ to $\pazocal{G}_{L,n}(k) = \pazocal{G}_{L,n}((r-k) + (2k-r))$, from which claim \eqref{lemma_cantor_set_start_at_k_bound_B_sufficient} follows.

\end{proof}
\label{lemma_non_shifted_iterates_subadditivity}
\end{lemma}

Now, combine inequality \eqref{shifted_integral_cantor_set_inequality} with Lemma \ref{lemma_non_shifted_iterates_subadditivity}, to conclude
\begin{align*}
    \lambda_k(\mathcal{C}_n(R)) \leq 2\lambda_0(\mathcal{C}_n(R)) \ \ \forall \ \ k,n\geq 0,
\end{align*}
which is a restatement of Theorem \ref{bound_operator_norm_thm1}. Theorem \ref{asymptotic_bound_operator_norm_thm2} follows by applying the estimates in Proposition \ref{proposition_asymptotic_estimates_lambda_0} to Theorem \ref{bound_operator_norm_thm1}. 


\addcontentsline{toc}{section}{References}
\bibliographystyle{unsrt}
\bibliography{bibliography}

\end{document}